\documentclass[12pt]{amsart}
\usepackage[colorlinks=true,pagebackref,hyperindex,citecolor=blue,linkcolor=red]{hyperref}
\usepackage{amsmath}
\usepackage{amsfonts}
\usepackage{amssymb}
\usepackage{setspace}
\usepackage{moreverb}
\usepackage[dvips]{graphicx} 
\usepackage[latin1]{inputenc}
\usepackage[T1]{fontenc}
\usepackage{amssymb}
\usepackage{tikz}
\usepackage{color}
\usepackage[all]{xy}
\usepackage{xfrac}
\usepackage[top=1in, bottom=1in, left=1in, right=1in]{geometry}
\usepackage{mathrsfs}

\usetikzlibrary{shapes}

\newtheorem{theoremx}{Theorem}

\newtheorem{theorem}{Theorem}[section]

\newtheorem{corollary}[theorem]{Corollary}
\newtheorem{lemma}[theorem]{Lemma}
\newtheorem{proposition}[theorem]{Proposition}

\theoremstyle{definition}
\newtheorem{definition}[theorem]{Definition}

\newtheorem{notation}[theorem]{Notation}

\newtheorem{remark}[theorem]{Remark}
\numberwithin{equation}{subsection}

\newcommand{\NN}{\mathbb{N}}

\newcommand{\ZZ}{\mathbb{Z}}

\newcommand{\FF}{\mathbb{F}}
\newcommand{\RR}{\mathbb{R}}

\newcommand{\cC}{\mathcal{C}}

\newcommand{\m}{\mathfrak{m}}
\newcommand{\fa}{\mathfrak{a}}

\newcommand{\fp}{\mathfrak{p}}

\newcommand{\End}{\operatorname{End}}

\newcommand{\Min}{\operatorname{Min}}

\allowdisplaybreaks

\begin{document}

\title{On $F$-thresholds of differential power filtrations}

\author[W. Badilla-C\'espedes]{W\'agner Badilla-C\'espedes}
\address{Escuela de Matemática, Universidad de Costa Rica, San José  11501, Costa Rica.}
\email{wagner.badillacespedes@ucr.ac.cr}
\email{wagner.badillacespedes@gmail.com}



\maketitle

\begin{abstract}
We study the $F$-threshold of differential power filtrations of monomial ideals.  We give the existence of this number and show that it is upper bounded by the maximum height of the minimal primes of the underlying monomial ideal. As an application of our results, we show that the dimension of the polynomial ring and the $F$-threshold of the differential power filtration of the monomial maximal ideal are the same. 
\end{abstract}
\section{Introduction}
The study of numerical invariants in prime characteristic has become a central topic in commutative algebra and singularity theory. Among these invariants, $F$-thresholds play a prominent role, since they measure the behavior of singularities. These numbers are obtained by comparing ordinary powers versus Frobenius powers. Given two ideals $I,J\subseteq R$ such that $I\subseteq \sqrt{J}$, the $F$-threshold of $I$ with respect to $J$ \cite{MTW2005,HMTW2008,DSNBP2018} is defined by $$c^J(I)=\lim\limits_{e\to\infty} \frac{\nu^J_I(p^e)}{p^e},$$ where  $\nu^J_I(p^e)=\max\{ t \in \NN | I^t \not\subseteq J^{[p^e]}\}$. These numbers still give information on $I,J$ and $R$. For instance, one can use $F$-thresholds to study integral closure, tight closure,  Hilbert-Samuel multiplicities \cite{HMTW2008},
 and $a$-invariants \cite{DSNB2018,DSNBP2018}. However, it is acknowledged by the community that the task of computing $c^J(I)$ is hard in general and few examples of explicit calculations are known, among them \cite{HDbinomial,HDpolynomials,Mu18,Tri,GVJVNB,SmVra,BCLC2025}. Recently, Koley and Kumar extended the notion of $F$-thresholds to filtrations of ideals, generalizing the classical concept \cite{KK}. They established the existence of $F$-thresholds for several natural filtrations, including symbolic power filtrations, and gave effective upper bounds.

 On the other hand, differential powers were introduced by Dao, De Stefani, Grifo, Huneke, and Núñez-Betancourt \cite{DDSGHNB2018}. These ideals have been of significant interest to study their relations with symbolic powers and to give versions of the Zariski--Nagata Theorem in which both concepts are involved \cite{DDSGHNB2018,DSGJ2020}. For prime ideals in polynomial rings over a perfect field, the differential powers coincide with symbolic powers, by the Zariski--Nagata Theorem \cite{Zariski1949,Nagata1975} (see also \cite[Proposition 2.14]{DDSGHNB2018}). However, if we do not assume perfect fields, then the Zariski?-Nagata Theorem may fail. For instance, it follows that $Q^{(2)}\not= Q^{\langle 2 \rangle}$ when $Q=(x^p-t) \subseteq \FF_p(t)[x]$ \cite[Example 3.8]{DSGJ2020}. This motivates the study of differential powers in rings of positive characteristic.

Motivated by the connections between differential powers and symbolic powers, together with the work of Koley and Kumar \cite{KK} on the $F$-threshold of symbolic power filtrations of ideals, in this article we study $F$-thresholds of differential power filtrations of monomial ideals. Specifically, we establish the existence of the $F$-thresholds for such filtrations and provide an explicit upper bound in terms of the height of the minimal primes of the underlying monomial ideal. 
\begin{theoremx}[{see Theorem \ref{theo:F-threshold of filtrations of differential powers}}] \label{Main-theo}
Let $I \subseteq J \subseteq R$ be two nonzero proper ideals, with $I$ monomial and $J$ radical. Then $\cC^{J}(I^{\langle \bullet \rangle})$ exists and $\cC^{J}(I^{\langle \bullet \rangle}) \leq \max\{\operatorname{ht}(\fp)\;|\;\fp \in \Min(I)\}$. 
\end{theoremx}

The previous result places differential power filtrations within the framework of $F$-thresholds, connecting differential operator techniques with invariants in prime characteristic. As a consequence, we recover the dimension of the ambient polynomial ring as the $F$-threshold of the differential power filtration of the monomial maximal ideal (see Corollary \ref{Coro:differential-filtration}).
\section{$F$-Thresholds of 
Filtrations of Ideals}
Throughout this section, $R$ denotes a ring of prime characteristic $p>0$. We introduce the definition and basic properties of the $F$-threshold of
a filtration of ideals. This is a recent invariant of singularities introduced by Koley and Kumar \cite{KK}, which extends the classical notion of the $F$-threshold of an ideal.

We begin by recalling the definition of a filtration of ideals. A filtration of ideals in $R$ is a sequence of ideals $\{\fa_i\}_{i \geq 0}$ satisfying the following conditions:
\begin{enumerate}
    \item $\fa_0=R$,
    \item $\fa_{i+1} \subseteq \fa_i$ for every $i \geq 0$,
    \item $\fa_i \fa_j \subseteq \fa_{i+j}$ for every $i,j\geq 0$.
\end{enumerate}
We denote the sequence $\{\fa_i\}_{i \geq 0}$ by $\fa_{\bullet}$. From the definition it follows that if $\fa_{\bullet}=\{\fa_i\}_{i \geq 0}$ is a filtration of ideals, then $\sqrt{\fa_{\bullet}}=\{\sqrt{\fa_i}\}_{i \geq 0}$ is also a filtration of ideals. Let $ \mathfrak{a}_\bullet = \{ \mathfrak{a}_i \}_{i \geq 0} $ and $ \mathfrak{b}_\bullet = \{ \mathfrak{b}_i \}_{i \geq 0} $ be two filtrations in $R$. We write $\fa_{\bullet} \leq \mathfrak{b}_{\bullet}$ if $\fa_i \subseteq \mathfrak{b}_i$ for every $i \geq 0$.

Now, we study the $F$-threshold of a filtration of ideals. To formalize this invariant, we adapt the classical notion of $F$-thresholds to the context of filtrations. 
\begin{definition}[{\cite{KK}}]
 Let $J\subseteq R$ be an ideal and $\fa_{\bullet}=\{\fa_i\}_{i \geq 0}$ be a filtration of ideals in $R$. For every $e \in \NN$, we define
 \begin{align*}
\nu_{\fa_{\bullet}}^{J}(p^e)=\sup\{t \in \NN\;|\; \fa_t \not\subseteq J^{[p^e]}\}.
 \end{align*}
We then define
 \begin{align*}
 \cC_{+}^{J}(\fa_\bullet) = \limsup_{e \to \infty} \frac{\nu^{J}_{\fa_{\bullet}}{(p^e)}}{p^{e}} \in \mathbb{R}_{\geq 0} \cup \{\pm \infty\} \;\text{and}\; 
 \cC_{-}^{J}(\fa_{\bullet}) = \liminf_{e \to \infty} \frac{\nu^{J}_{\fa_{\bullet}}{(p^e)}}{p^{e}} \in \mathbb{R}_{\geq 0} \cup \{\pm \infty\}.
 \end{align*}
 If $\cC_{+}^{J}(\fa_\bullet) = \cC_{-}^{J}(\fa_{\bullet}) \in \RR_{\geq 0}$, then we denote it by $ \cC^{J}(\fa_{\bullet})$ and call it the $F$-threshold of $\fa_{\bullet}$ with respect to $J$. When $(R,\m,K)$ is either  a local ring or a standard graded
$K$-algebra, we refer to $\cC^{\m}(\fa_{\bullet})$ as the $F$-threshold of $\fa_{\bullet}$.
\end{definition}

On the other hand, $\cC_{\pm}^{J}(\fa_{\bullet})$ refers to both the upper and lower limits of the sequence $\left\{\displaystyle \frac{\nu_{\mathfrak{a}_{\bullet}}^J\left(p^e\right)}{p^e}\right\}_{e \in \NN}$. Furthermore, $\mathcal{C}_{ \pm}^J\left(\mathfrak{a}_{\bullet}\right) \leq \mathcal{C}_{ \pm}^J\left(\mathfrak{b}_{\bullet}\right)$ indicates that both components satisfy the requirements
\begin{align*}
    \mathcal{C}_{ -}^J\left(\mathfrak{a}_{\bullet}\right) \leq \mathcal{C}_{ -}^J\left(\mathfrak{b}_{\bullet}\right) \; \text{and}\; \mathcal{C}_{+}^J\left(\mathfrak{a}_{\bullet}\right) \leq \mathcal{C}_{+}^J\left(\mathfrak{b}_{\bullet}\right).
\end{align*}

The following proposition describes the behavior of the $F$-threshold of a filtration of ideals when we have an $F$-pure ring.

\begin{proposition}[{\cite[Proposition $3.3$]{KK}}] \label{Prop: filtration-F-puro}
 Let $J\subseteq R$ be a nonzero proper ideal and $\mathfrak{a}_\bullet$ be a filtration in $R$. Let $e_0 \geq 0$ be a fixed integer. If $J^{\left[p^e\right]}=\left(J^{\left[p^e\right]}\right)^F$ for all $e \geq e_0$, then $p \nu_{\fa_{\bullet}}^J\left(p^e\right) \leq \nu_{\fa_{\bullet}}^J\left(p^{e+1}\right)$ for all $e \geq e_0$. Moreover,$$\mathcal{C}_{ \pm}^J\left(\mathfrak{a}_{\bullet}\right)=\lim _{e \rightarrow \infty} \frac{\nu_{\mathfrak{a}_{\bullet}}^J\left(p^e\right)}{p^e}=\sup _{e \geq e_0} \frac{\nu_{\mathfrak{a}_{\bullet}}^J\left(p^e\right)}{p^e}.
$$ In particular, when $R$ is $F$-pure, $$
\mathcal{C}_{ \pm}^J\left(\mathfrak{a}_{\bullet}\right)=\lim _{e \rightarrow \infty} \frac{\nu_{\mathfrak{a}_{\bullet}}^J\left(p^e\right)}{p^e}=\sup _{e \geq 0} \frac{\nu_{\mathfrak{a}_{\bullet}}^J\left(p^e\right)}{p^e}.
$$
\end{proposition}
We note that Proposition \ref{Prop: filtration-F-puro} does not guarantee the existence of $\mathcal{C}_{ \pm}^J\left(\mathfrak{a}_{\bullet}\right)$, since the limit of the sequence $\left\{\displaystyle \frac{\nu_{\mathfrak{a}_{\bullet}}^J\left(p^e\right)}{p^e}\right\}_{e \in \NN}$ may be equal to $+\infty$. 

The following theorem provides a comparison between the $F$-thresholds of two filtrations, assuming that one contains the other.
\begin{theorem}[{\cite[Theorem 3.5]{KK}}] \label{Theo: inequality-filtration}
Let $ \mathfrak{a}_\bullet = \{ \mathfrak{a}_i \}_{i \geq 1} $ and $ \mathfrak{b}_\bullet = \{ \mathfrak{b}_i \}_{i \geq 1} $ be two filtrations in $R$. Let $J \subseteq R$ be an ideal. If $ \mathfrak{a}_\bullet \leq \mathfrak{b}_\bullet$, then $\mathcal{C}^{J}_{\pm}(\mathfrak{a}_\bullet) \leq \mathcal{C}^{J}_{\pm}(\mathfrak{b}_\bullet)$. 
\end{theorem}

\begin{remark}\label{remark-inequality}
Let $R$ be an $F$-pure ring. Let $J \subseteq R$ be a nonzero proper ideal. Let $ \mathfrak{a}_\bullet = \{ \mathfrak{a}_i \}_{i \geq 1}$ and $ \mathfrak{b}_\bullet = \{ \mathfrak{b}_i \}_{i \geq 1}$ be two filtrations in $R$ such that  $ \mathfrak{a}_\bullet \leq \mathfrak{b}_\bullet$. By Proposition \ref{Prop: filtration-F-puro} and Theorem \ref{Theo: inequality-filtration}, we have
$$\cC_{\pm}^{J}(\fa_{\bullet}) = \sup_{e \geq 0} \frac{\nu_{\fa_{\bullet}}^{J}(p^{e})}{p^{e}} \leq \sup_{e \geq 0} \frac{\nu_{\mathfrak{b}_{\bullet}}^{J}(p^{e})}{p^{e}} = \cC_{\pm}^{J}(\mathfrak{b}_{\bullet}).$$ This inequality guarantees the existence of $\cC^{J}(\fa_{\bullet})$, provided that $\cC^{J}(\mathfrak{b}_{\bullet})$ exists.
\end{remark}

\section{Existence of $F$-Thresholds of Differential Power Filtrations in the Monomial Case}
 Let $R$ be a finitely generated $K$-algebra. We define the $K$-linear differential operators of order less than or equal to $n$ in $R$ by
\begin{enumerate}
\item $D^0_R=R=\End_R(R) \subseteq \End_K(R)$,
\item $D^{n}_R=\{\delta \in \End_K(R) \;|\;\delta f-f \delta \in D^{n-1}_R\; $for every$ \;f \in R\}$ for every $n \geq 1$.
\end{enumerate}
The ring of $K$-linear differential operators is defined by $$D_R=\bigcup_{n \in \NN}D^n_R.$$ 

Suppose that $R$ is either the polynomial ring $K[x_1,\ldots, x_d]$ or the formal power series ring $K[[x_1, \ldots, x_d]]$. The ring of $K$-linear differential operators is $$D_R=R\left\langle \frac{1}{j!}\frac{\partial^j}{\partial x_i^j} \; \bigg| \; i=1,\ldots,d; \;j\in \NN \right\rangle,$$ that is, the free $R$-module generated by the differential operators $\displaystyle  \frac{1}{j_1!}\frac{\partial^{j_1}}{\partial x_1^{j_1}} \cdots  \frac{1}{j_d!}\frac{\partial^{j_d}}{\partial x_d^{j_d}}$ \cite{G67}. In this case $$D_R^n=R\left\langle \frac{1}{\alpha_1!}\frac{\partial^{\alpha_1}}{\partial x_1^{\alpha_1}} \cdots \frac{1}{\alpha_d!}\frac{\partial^{\alpha_d}}{\partial x_d^{\alpha_d}}\; \bigg| \;  \alpha_1+\ldots +\alpha_d \leq n \right\rangle.$$
 
\begin{definition}[\cite{DDSGHNB2018}]\label{Def:differential-power}
Let $R$ be a finitely generated $K$-algebra. Let $I \subseteq R$ be an ideal, and let $n$ be a
positive integer. We define the $n$-th $K$-linear differential power of $I$ by $$I^{\langle n\rangle}=\{f \in R\;|\; \delta(f) \in I\; \text{for all}\; \delta \in D_{R}^{n-1}\}.$$
\end{definition}
For $n>0$, the differential power $I^{\langle n\rangle}$ is an ideal in $R$ \cite{DDSGHNB2018}. By convention, $I^{\langle 0\rangle}=R$. We can note that $I^{\langle n+1\rangle} \subseteq I^{\langle n\rangle}$ since $D_R^{n-1} \subseteq D_R^{n}$. 

For the inclusion $I^{\langle n\rangle} I^{\langle m\rangle} \subseteq I^{\langle n+m\rangle}$, Kenkel, McPherson, Page, Smolkin, Stephenson, and Yang provide examples of families of ideals in the polynomial ring over a field of characteristic zero that satisfy this property \cite[Lemma 2.9]{KMPSSY2021}. Moreover, for $\ZZ$-linear differential operators, De Stefani, Grifo, and Jeffries show that the differential powers of prime ideals in the polynomial ring coincide with their symbolic powers, expanding the examples where the inclusion we are discussing holds \cite[Theorem 3.3]{DSGJ2025}. On the other hand, this inclusion fails in general  \cite[Example 4.35]{BJNB2019}. However, it is natural to ask where such an inclusion might be realized, for instance, in normal domains \cite[Question 4.36]{BJNB2019}. This discussion motivates the following definition.

\begin{definition}\label{Def:differential-filtration}
Let $R$ be a finitely generated $K$-algebra. Let $I \subseteq R$ be a nonzero ideal. The sequence of ideals $\{I^{\langle i \rangle }\}_{i \geq 0}$ is called an $I$-differential power filtration if $I^{\langle i \rangle }I^{\langle j \rangle } \subseteq I^{\langle i+j \rangle }$ for every $i,j \geq 0$. In this case, we denote the sequence $\{I^{\langle i \rangle }\}_{i \geq 0}$ by $I^{\langle \bullet \rangle }$.
\end{definition}

The following proposition tells us that, in the polynomial ring, we always have differential power filtrations of ideals.

\begin{proposition}\label{Prop:differential-filtration-polinomial}
Let $R=K[x_{1},\ldots ,x_{d}]$ be a polynomial ring where $K$ is an arbitrary field, and let $I \subseteq R$ be an ideal. Then $I^{\langle n \rangle }I^{\langle m \rangle } \subseteq I^{\langle n+m \rangle }$ for every $n,m \geq 0$.   
\end{proposition} 
\begin{proof} 
In order to prove this proposition, for each $\gamma  \in \NN^d$ we denote $|\gamma|=\gamma_1 + \ldots + \gamma_d$, and $$\delta_{\gamma}=\displaystyle \frac{1}{\gamma_{1}!} \frac{\partial^{\gamma_1}}{\partial x_1^{\gamma_1}}
\cdots
\frac{1}{\gamma_{d}!} \frac{\partial^{\gamma_d}}{\partial x_d^{\gamma_d}}.$$ 

Let $s \in I^{\langle n \rangle }$ and $t \in I^{\langle m \rangle }$. We want to show that $st \in I^{\langle n+m \rangle }$. Since $\displaystyle D_R= \bigoplus_{\gamma \in \NN^d}R  \cdot \delta_\gamma$, it suffices to verify that $\delta_{\alpha}(st) \in I$ for every $\alpha \in \NN^d$ such that $|\alpha|\leq n+m-1$. Independently of the characteristic of the field $K$, the family of operators $\{\delta_{\gamma}\}_{\gamma \in \NN^d}$ is a Hasse--Schmidt derivation on $R$ (see \cite[Example 2.2.2]{Traves1998}). Then the operator $\delta_\alpha$ satisfies the Leibniz rule $$\delta_{\alpha}(st)=\sum_{\beta+\omega=\alpha} \delta_{\beta}(s)\delta_{\omega}(t).$$

If $|\beta|\leq n-1$, then $\delta_{\beta}(s) \in I$, and so $\delta_{\beta}(s)\delta_{\omega}(t) \in I$. If instead $|\beta|>n-1$, then $$n-1<|\beta|=|\alpha-\omega|=|\alpha|-|\omega|\leq n+m-1-|\omega|,$$ which implies $|\omega|\leq m-1$. As a consequence, $\delta_{\omega}(t) \in I$, and thus again $\delta_{\beta}(s)\delta_{\omega}(t) \in I$.

From the above it follows that $\delta_{\alpha}(st) \in I$.
\end{proof}

Let $R=K[x_{1},\ldots ,x_{d}]$ be a polynomial ring where $K$ is an arbitrary field, and let $I \subseteq R$ be an ideal. Then the previous proposition tells us that $I^{\langle \bullet \rangle}=\{I^{\langle i \rangle}\}_{i\geq 0}$ is an $I$-differential power filtration.

Throughout the remainder of this article we shall take 
$R$ to be the polynomial ring 
$K[x_{1},\ldots ,x_{d}]$ with $K$ a field of prime characteristic $p$. We focus on studying the $F$-threshold of an $I$-differential power filtration when $I$ is a monomial ideal in $R$.

\begin{notation}
Let $I \subseteq R$ be an ideal. We denote the minimal number of generators of $I$ by $\mu(I)$, and we denote the height of $I$ by $\operatorname{ht}(I)$.    
\end{notation}

\begin{lemma}\label{lemma:monomial-prime}
Let $\fp \subseteq R$ be a monomial prime ideal. Then the following statements hold.
\begin{enumerate}
\item For every $n \in\NN$, $\fp^{\langle n \rangle } = \fp^{n}$.
\item Let $N \in \NN$ such that $N \geq \mu(\fp)$. Then $\fp^{\langle Np^{e}\rangle } \subseteq \fp^{[p^{e}]}$ for every $e \in \NN$.
\end{enumerate}
\end{lemma}
\begin{proof}
Given that $\fp$ is a monomial prime ideal, we have $\fp=(x_{j_1},\ldots, x_{j_m})$, and we set $S=\{j_1,\ldots,j_m\}$.

For the first part of the lemma, we observe that by \cite[Proposition 2.5]{DDSGHNB2018}, it suffices to show that $\fp^{\langle n \rangle } \subseteq \fp^{n}$. Let $f \not \in \fp^{n}$. Hence $f$ has a term of the form $\lambda x_1^{\alpha_1}\cdots x_d^{\alpha_d}$ with $\lambda$ a nonzero element of $K$ and $\displaystyle \sum_{i\in S} \alpha_i \leq n-1$. Consider the differential operator $\delta=\displaystyle \frac{1}{\omega_{1}!} \frac{\partial^{\omega_1}}{\partial x_1^{\omega_1}}
\cdots
\frac{1}{\omega_{d}!} \frac{\partial^{\omega_d}}{\partial x_d^{\omega_d}}$, where $\omega_i=\alpha_i$ whenever $i \in S$ and $\omega_i=0$ otherwise.  Applying $\delta$ to $f$, it maps $\lambda x_1^{\alpha_1}\cdots x_d^{\alpha_d}$ to $\lambda \displaystyle\prod_{i \not \in S}x_i^{\alpha_i}$, and any
other term appearing in $f$ to terms that differ from the additive inverse of $\lambda \displaystyle\prod_{i \not \in S}x_i^{\alpha_i}$. In this way, we obtain $\delta(f) \not \in \fp$. Moreover, we note that $\delta \in D_R^{n-1}$. As a consequence, $f \not \in \fp^{\langle n \rangle }$.  

The second part is a consequence of Part (1), in view of the inclusions
 $$\fp^{Np^{e}} \subseteq \fp^{\mu(\fp)p^{e}} \subseteq \fp^{[p^{e}]}.$$
\end{proof}

We are now prepared to prove the main result of this manuscript, which appears in the introduction as Theorem \ref{Main-theo}. This theorem establishes the existence of the $F$-threshold of an $I$-differential power filtration for the monomial case and gives us an explicit upper bound expressed via the minimal primes of $I$.

\begin{theorem}\label{theo:F-threshold of filtrations of differential powers}

Let $I \subseteq J \subseteq R$ be two nonzero proper ideals, with $I$ monomial and $J$ radical. Then $\cC^{J}(I^{\langle \bullet \rangle})$ exists and $\cC^{J}(I^{\langle \bullet \rangle}) \leq \max\{\operatorname{ht}(\fp)\;|\;\fp \in \Min(I)\}$. 
\end{theorem}
\begin{proof}
Since $I \subseteq \sqrt{I}$, we have $I^{\langle \bullet \rangle} \leq \sqrt{I}^{\langle \bullet \rangle}$. By Remark \ref{remark-inequality}, we may assume, without loss of generality, that $I$ is a radical ideal. Let $\fp_1,\ldots, \fp_{\ell}$ be the minimal prime ideals of $I$. We take $N=\max\{\mu(\fp)\;|\;\fp \in \Min(I)\}$. By Lemma \ref{lemma:monomial-prime} (2), we have 
\begin{align*}
I^{\langle Np^{e} \rangle}&=\left(\fp_1 \cap \ldots \cap \fp_{\ell}\right)^{\langle Np^{e} \rangle} \\
&=\fp_1^{\langle Np^{e} \rangle}\cap \ldots \cap \fp_{\ell}^{\langle Np^{e} \rangle} \\
& \subseteq \fp_1^{[p^{e}]}\cap \ldots \cap \fp_{\ell}^{[ p^{e}]}\\
& = \left(\fp_1\cap \ldots \cap \fp_{\ell} \right)^{[ p^{e}]}\\
&= I^{[p^{e}]} \subseteq J^{[p^{e}]},
\end{align*}
 for every $e \in \NN$. As a consequence, for every $e \in \NN$ $$\frac{\nu^{J}_{I^{\langle \bullet \rangle}}{(p^e)}}{p^{e}} \leq \frac{Np^{e}-1}{p^{e}}.$$ Now, applying lim sup and lim inf we get that $0 \leq \mathcal{C}_{\pm}^{J}\left(I^{\langle \bullet \rangle}\right) \leq N$. By Proposition \ref{Prop: filtration-F-puro}, $$\mathcal{C}_{ \pm}^{J}\left(I^{\langle \bullet \rangle}\right)=\lim _{e \rightarrow \infty} \frac{\nu_{I^{\langle \bullet \rangle}}^{J}\left(p^e\right)}{p^e}=\sup _{e \geq 0} \frac{\nu_{I^{\langle \bullet \rangle}}^{J}\left(p^e\right)}{p^e}.$$ Since $\left\{\displaystyle \frac{\nu_{I^{\langle \bullet \rangle}}^{J}\left(p^e\right)}{p^e}\right\}_{e \in \NN}$ is upper bounded by $N$, $\mathcal{C}^{J}\left(I^{\langle \bullet \rangle}\right)$ exists and $$\mathcal{C}^{J}\left(I^{\langle \bullet \rangle}\right)  \leq N = \max\{\operatorname{ht}(\fp)\;|\;\fp \in \Min(I)\}.$$
\end{proof}

We end this section with a differential power filtration, which $F$-threshold coincides with the dimension of $R$.
\begin{corollary} \label{Coro:differential-filtration}
Let $\m \subseteq R$ be the ideal generated by $x_1, \ldots, x_d$. Then $\cC^{\m}(\m^{\langle \bullet \rangle})=d$. 
\end{corollary}
\begin{proof}
By virtue of Theorem \ref{theo:F-threshold of filtrations of differential powers}, it suffices to show that $d \leq \cC^{\m}(\m^{\langle \bullet \rangle})$. We take $s=d(p^{e}-1)$.  We have $(x_1 \cdots x_d)^{p^{e}-1} \in \m^s$, but $(x_1 \cdots x_d)^{p^{e}-1} \not \in \m^{[p^{e}]}$. By Lemma \ref{lemma:monomial-prime} (1), $\m^{\langle s \rangle} \not \subseteq \m^{[p^{e}]}$. Hence, for every $e \in \NN$ $$\frac{d(p^{e}-1)}{p^{e}} \leq \frac{\nu^{\m}_{\m^{\langle \bullet \rangle}}{(p^e)}}{p^{e}},$$ which take us to $d \leq \cC^{\m}(\m^{\langle \bullet \rangle})$ by  taking the limit as $e \rightarrow \infty$.
\end{proof}

\section*{Acknowledgments}
I thank Luis Núñez-Betancourt for helpful conversations and comments. I also thank the anonymous referee for useful comments and suggestions.



\bibliographystyle{alpha}
\newcommand{\etalchar}[1]{$^{#1}$}

\end{document}